\documentclass[paper=a4,11pt,headings=small]{scrartcl}

\usepackage[T1]{fontenc}
\usepackage[utf8]{inputenc}
\usepackage{lmodern}
\usepackage{graphicx}

\usepackage{natbib}
\usepackage{hyperref}

\usepackage{url}

\usepackage{amsmath,amsfonts,amssymb,amsthm}
\usepackage{mathtools}
\usepackage{todonotes}

\usepackage{hyperref}
\hypersetup{
  pdftitle={The Kinematic Image of 2R Dyads and Exact Synthesis of 5R Linkages},  pdfauthor={Hans-Peter Schröcker},  pdfkeywords={2R dyad, kinematic map, 5R linkage, Goldberg linkage, linkage synthesis, dual quaternions},  hidelinks=true,}

\setkomafont{sectioning}{\normalfont\bfseries}

\newtheorem{theorem}{Theorem}
\newtheorem{lemma}{Lemma}
\theoremstyle{definition}
\newtheorem{remark}{Remark}

\newcommand{\eps}{\varepsilon}
\newcommand{\D}{\mathbb{D}}
\renewcommand{\H}{\mathbb{H}}
\newcommand{\C}{\mathbb{C}}
\newcommand{\R}{\mathbb{R}}
\newcommand{\qi}{\mathbf{i}}
\newcommand{\qj}{\mathbf{j}}
\newcommand{\qk}{\mathbf{k}}
\newcommand{\ci}{\mathrm{i}}
\newcommand{\cj}[1]{\overline{#1}}
\newcommand{\Norm}[1]{\Vert{#1}\Vert}
\newcommand{\quadric}[1]{\mathcal{#1}}
\newcommand{\QQ}{\quadric{Q}}
\newcommand{\SQ}{\quadric{S}}
\newcommand{\EG}{E}

\newcommand{\NC}{\quadric{N}}
\newcommand{\SE}[1][3]{\mathrm{SE}(#1)}
\newcommand{\qf}{\omega}

\title{The Kinematic Image of 2R Dyads\\and Exact Synthesis of 5R Linkages}
\author{Tudor-Dan Rad and Hans-Peter Schröcker\\
  Unit Geometry and CAD\\
  University of Innsbruck, Austria}

\begin{document}

\maketitle

\begin{abstract}
    We characterise the kinematic image of the constraint variety of a
  2R dyad as a regular ruled quadric in a three-space that contains a
  ``null quadrilateral''. Three prescribed poses determine, in
  general, two such quadrics. This allows us to modify a recent
  algorithm for the synthesis of 6R linkages in such a way that two
  consecutive revolute axes coincide, thus producing a 5R
  linkage. Using the classical geometry of twisted cubics on a
  quadric, we explain some of the peculiar properties of the the
  resulting synthesis procedure for 5R linkages.
 \end{abstract}

\section{Introduction}
\label{sec:introduction}

Given three poses of a rigid body, there exists, in general, a unique
closed kinematic loop of four revolute joints with one degree of
freedom, a so-called Bennett or 4R linkage, one of whose links visits
the three poses. \cite{suh69} solved this synthesis problem by
constructing two 2R dyads to the three given poses and proving that
they can be combined to form a Bennett linkage. The three pose
synthesis problem for 4R linkages was re-visited by
\cite{brunnthaler05:_bennet_synthesis} in a dual quaternion setting. A
generalisation of their approach led to a factorisation theory for
``motion polynomials'' \citep{hegedus13:_factorization2}. It was used
by \cite{hegedus14:_four_pose_synthesis} to synthesise spatial 6R
linkages -- closed kinematic chains with six revolute joints -- to
four prescribed poses in general position. More precisely, the four
poses determine two families of (not necessary real) rational coupler
motions, each giving rise to nine families of overconstrained 6R
linkages. It should also be mentioned that the construction is only
capable of producing a special type of 6R linkages.

In this paper we intend to close the gap in this sequence of exact
factorisation based synthesis procedures for closed loop linkages and
present a construction of 5R linkages (Goldberg linkages) to ``more
than three but less than four poses''. Our synthesis algorithm is
based on a characterisation of the image of open 2R chains under
Study's kinematic mapping and the observation that the coupler motion
of a 5R linkage is also parameterised by a cubic motion polynomial
\citep[see][]{hegedus13:_bonds2}. This allows to treat 5R
synthesis as a special case of above mentioned synthesis of 6R
linkages.

We continue this text with a quick introduction to dual quaternions
and kinematics in \autoref{sec:preliminaries}. In
\autoref{sec:characterisation} we provide a geometric characterisation
for the kinematic image of a 2R dyad. While necessary properties can
be derived easily, the proof of their sufficiency requires more work
and comprises the main part of this text. The synthesis of 5R linkage
in \autoref{sec:synthesis} is then merely a corollary to previous
results. Note however, that it is not obvious how to use the additional
degrees of freedom in exact 5R synthesis.

\section{Preliminaries}
\label{sec:preliminaries}

This article's scene is the projectivised dual quaternion model of spatial
kinematics. Here, we give a very brief introduction to this model for
the purpose of settling our notation. More details will be introduced in the
text as needed. For more thorough introductions to dual quaternions
and there relations to kinematics we refer to
\cite[Section~3]{klawitter15} or
\cite[Section~11]{selig05:_geometric_fundamentals_robotics}. Familiarity
with \cite{hegedus13:_factorization2} is recommended too.

The dual quaternions, denoted by $\D\H$, form an associative algebra
in $\R^8$ where multiplication of the base elements $1$, $\qi$, $\qj$,
$\qk$, $\eps$, $\eps\qi$, $\eps\qk$, and $\eps\qk$ is defined by the
rules
\begin{equation*}
  \qi^2 = \qj^2 = \qk^2 = \qi\qj\qk = -1,\quad
  \eps^2 = 0,\quad
  \qi\eps = \eps\qi,\quad
  \qj\eps = \eps\qj,\quad
  \qk\eps = \eps\qk.
\end{equation*}
An element $q \in \D\H$ may be written as $q = p + \eps d$ with
quaternions $p,d \in \H \coloneqq \langle 1, \qi, \qj, \qk \rangle$
(angled brackets denote linear span).  In this case the
\emph{quaternions} $p$ and $d$ are referred to as \emph{primal} and
\emph{dual} part of $q$, respectively. The conjugate dual quaternion
is $\cj{q} = \cj{p} + \eps\cj{d}$ and conjugation of quaternions is
done by multiplying the coefficients of $\qi$, $\qj$, and $\qk$ with
$-1$. It satisfies the rule $\cj{qr} = \cj{r}\,\cj{q}$ for any
$q,r \in \D\H$. The dual quaternion norm is defined as
$\Norm{q} = q\cj{q}$. We readily verify that it is a \emph{dual
  number,} that is, an element of
$\D \coloneqq \langle 1,\eps \rangle$.

We identify linearly dependent non-zero dual quaternions and thus
arrive at the projective space $P^7 = P(\R^8)$. Writing $[q]$ for the
point in $P^7$ that is represented by $q \in \D\H$, the \emph{Study
  quadric} is defined as
$\SQ \coloneqq \{[q] \in P^7\colon \Norm{q} \in \R \}$.  With
$q = p + \eps d$, the algebraic condition for $[q] \in \SQ$ is
$p\cj{d} + d\cj{p} = 0$.

Identifying $P^3$ with the projective subspace generated by
$\langle 1, \eps\qi, \eps\qj, \eps\qk \rangle$, a point
$[q] = [p + \eps d] \in \SQ$ with non-zero primal part acts on
$[x] \in P^3$ via
\begin{equation}
  \label{eq:1}
  [x] \mapsto [(p + \eps d) x (\cj{p} - \eps \cj{d})].
\end{equation}
The map \eqref{eq:1} is the projective extension of a rigid body
displacement in $\R^3$. Composition of displacements corresponds to
dual quaternion multiplication. This isomorphism between $\SE$, the
group of rigid body displacements, and the projectivisation of the
group of dual quaternions of real norm and non-zero primal part
provides a rich and solid algebraic and geometric basis for
investigations in kinematics.

\section{A characterisation of 2R spaces}
\label{sec:characterisation}

We complement the Study quadric by the quadric
$\NC = \{ [q] \colon \Norm{q} \in \eps\R \}$ of points represented by
dual quaternions whose norm has vanishing primal part. It is a quadric
of rank four. The set $\EG = \{[q] \colon q \in \eps\H\}$ of its
singular points is a projective space of dimension three and it is
contained in the Study quadric. Only the points of $\EG$ are real
points of $\NC$ whence we are led to consider the complex extension
$P(\C^8)$ of $P^7$. We call $\NC$ the \emph{null cone} and refer to
the set $\EG$ of its singular points as the \emph{exceptional three
  space.}

Given two non coplanar lines $\ell_1$, $\ell_2$ in Euclidean three
space, we consider the set of all displacements obtained as
composition of a rotation around $\ell_2$, followed by a rotation
about $\ell_1$. Its kinematic image is known to lie in a three space
\cite[Section~11.4]{selig05:_geometric_fundamentals_robotics} which we
call a \emph{2R space}. By construction, a 2R space contains the point
$[1]$ (the identity displacement) and intersects $\SQ$ in a doubly
ruled regular quadric. The rulings are obtained by varying one
revolute angle and fixing the other. In order to characterise 2R
spaces among all three-dimensional subspaces of $P^7$ with these
properties, we introduce the notion of \emph{null quadrilaterals.}
These are spatial quadrilaterals contained in both, $\SQ$ and
$\NC$. We call their edges \emph{null lines.}

\begin{theorem}
  \label{th:1}
  A three space $P \subset P^7$ with $[1] \in P$ is a 2R space if and
  only if it
  \begin{itemize}
  \item intersects the Study quadric in a regular ruled quadric~$\QQ$,
  \item does not intersect the exceptional three space $\EG$, and
  \item contains a null quadrilateral.
  \end{itemize}
\end{theorem}

The first and second item in \autoref{th:1} exclude exceptional cases
with coplanar, complex, or ``infinite'' revolute axes, the latter
corresponding to prismatic joints. The important point is existence of
a null quadrilateral. We split the proof of \autoref{th:1}
into a series of lemmas. It will be finished by the end of this
section.

\begin{lemma}
  \label{lem:1}
  The straight line $[x] \vee [y]$ is contained in $\SQ \cap \NC$ if
  and only if $x\cj{x} = y\cj{y} = x\cj{y} + y\cj{x} = 0$.
\end{lemma}
We omit the straightforward computational proof but note that the
left-hand sides of each of the three conditions in \autoref{lem:1} are
dual numbers. Hence, they give \emph{just six} independent conditions on
the real coefficients of $x$ and $y$.

\begin{lemma}
  \label{lem:2}
  The conditions of \autoref{th:1} are necessary for 2R spaces.
\end{lemma}

\begin{proof}
  The constraint variety of a 2R chain can be parameterised as
  \begin{equation}
    \label{eq:2}
    R(t_1,t_2) = (t_1 - h_1)(t_2-h_2)
  \end{equation}
  with two dual quaternions $h_1, h_2$ that satisfy
  $h_1\cj{h_1} = h_2\cj{h_2} = 1$, $h_1+\cj{h_1} = h_2+\cj{h_2} = 0$
  and $h_1h_2 \neq h_2h_1$. The first condition ensures that the dual
  quaternions are suitably normalised and $[h_1], [h_2] \in \SQ$.  The
  second condition means that $h_1$ and $h_2$ describe half-turns
  (rotations through an angle of $\pi$). The third condition ensures
  that the axes of these half turns are
  different. Equation~\eqref{eq:2} describes a composition of two
  rotations for all values of $t_1$, $t_2$ in $\R \cup \{\infty\}$
  with $\infty$ corresponding to zero rotation angle. Even if their
  kinematic meaning is unclear, we also allow complex parameter
  values. Expanding \eqref{eq:2} yields
  \begin{equation*}
    R(t_1,t_2) = t_1t_2 - t_1h_2 - t_2h_1 + h_1h_2.
  \end{equation*}
  We see that the kinematic image of the 2R dyad lies in the three
  space spanned by $[1]$, $[h_1]$, $[h_2]$, $[h_1h_2]$. In a suitable
  projective reference with these points as base points, the surface
  parameterised by \eqref{eq:2} is the quadric with equation
  $x_0x_3 - x_1x_2 = 0$ which is indeed regular and ruled.

  The intersection of $P$ with the exceptional three space $\EG$ is
  non empty if and only if the primal part of $R(t_1,t_2)$ vanishes
  for certain parameter values $t_1$, $t_2$. This can only happen if
  the primal parts of $h_1$ and $h_2$ are linearly dependent over $\R$
  but then the revolute axes are parallel and $P$, contrary to our
  assumption, is contained in $\SQ$. The other possibility for
  $P \subset \SQ$, intersecting revolute axes, has been excluded as
  well. Clearly, $P$ is not contained in the null cone $\NC$ either.
  We claim that the intersection of $P$ and $\NC$ consists of the four
  lines given by $t_1 = \pm\ci$, $t_2 = \pm\ci$.  Indeed, they are
  null lines. Take for example $x = (\ci - h_1)(t_2-h_2)$ and, in view
  of \autoref{lem:1}, compute
  \begin{equation*}
    \begin{aligned}
      x\cj{x} &= (\ci-h_1)(t_2-h_2)(t_2-\cj{h_2})(\ci-\cj{h_1}) = (t_2-h_2)(t_2-\cj{h_2})(\ci-h_1)(\ci-\cj{h_1}) \\
              &= (t_2-h_2)(t_2-\cj{h_2})(-\ci^2-\ci(h_1+\cj{h_1})+h_1\cj{h_1}) = (t_2-h_2)(t_2-\cj{h_2})(-1-0+1) = 0.
    \end{aligned}
  \end{equation*}
  This derivation uses the fact that $(t_2-h_2)(t_2-\cj{h_2})$ is a
  real number and thus commutes with all other factors. The cases
  $t_1=-\ci$, $t_2=\pm\ci$ are similar so that we have verified all
  conditions of \autoref{th:1}.
\end{proof}

The proof of sufficiency is more involved. We need two additional
lemmas from projective geometry which are formulated and proved in
\autoref{sec:appendix}.

\begin{lemma}
  \label{lem:3}
  A three space $P$ that satisfies all conditions of \autoref{th:1} is
  a 2R space.
\end{lemma}

\begin{proof}
  Denote the vertices of the null quadrilateral (in that order) by
  $[u_1]$, $[v_1]$, $[u_2]$, and $[v_2]$. Intersecting the Study null
  quadrilateral with the tangent hyperplane of $\SQ$ at $[1]$ yields
  points
  \begin{equation*}
    [m_1] \in [u_1] \vee [v_1],\quad
    [n_1] \in [v_1] \vee [u_2],\quad
    [m_2] \in [u_2] \vee [v_2],\quad
    [n_2] \in [v_2] \vee [u_1]
  \end{equation*}
  of a planar and non-degenerate quadrilateral
  (\autoref{fig:null-lines}). The joins $[m_1] \vee [m_2]$ and
  $[n_1] \vee [n_2]$ are real lines and their real points correspond
  to rotations about fixed axes. Pick real rotation quaternions
  $[h_1] \in [m_1] \vee [m_2]$ and $[h_2] \in [n_1] \vee [n_2]$.  The
  axes of these rotations are the only candidates for our 2R
  dyad. Hence, we have to show that either $[h_1h_2] \in P$ or
  $[h_2h_1] \in P$. It is easy to see that this is the case if and only
  if either $[m_1n_1] \in P$ or $[n_1m_1] \in P$.  In fact, we will
  even show that one of these products equals~$v_1$.

  \begin{figure}
    \centering
    \includegraphics{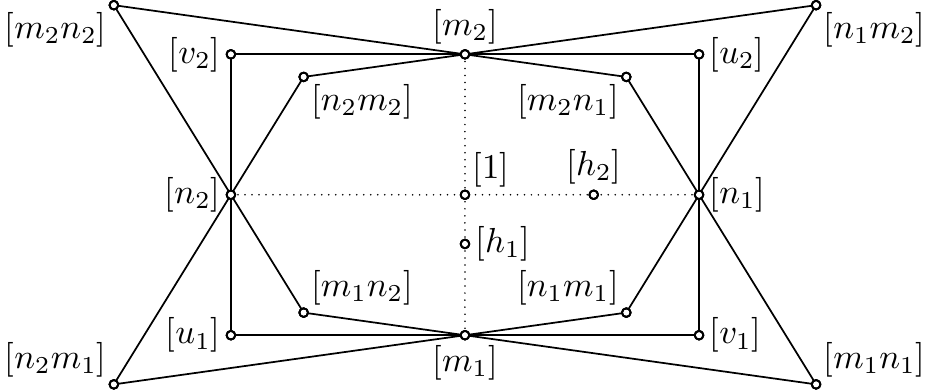}
    \caption{Null lines and null quadrilaterals in the proof of
      \autoref{lem:3}.}
    \label{fig:null-lines}
  \end{figure}

  We claim that both $M_1 \coloneqq [m_1] \vee [m_1n_1]$ and
  $M_2 \coloneqq [m_1] \vee [n_1m_1]$ are null lines. In order to show
  this, we have to verify the conditions of \autoref{lem:1}:
  \begin{gather*}
    (m_1n_1)\cj{(m_1n_1)} = m_1(\underbrace{n_1\cj{n_1}}_{=0})\cj{m_1} = 0,\quad
    (n_1m_1)\cj{(n_1m_1)} = n_1(\underbrace{m_1\cj{m_1}}_{=0})\cj{n_1} = 0,\\
    m_1\cj{(m_1n_1)} + (m_1n_1)\cj{m_1} =
    m_1(\underbrace{n_1 + \cj{n_1}}_{\in \C})\cj{m_1} = (n_1+\cj{n_1})(\underbrace{m_1\cj{m_1}}_{=0}) = 0,\\
    m_1\cj{(n_1m_1)} + (n_1m_1)\cj{m_1} =
    (\underbrace{m_1\cj{m_1}}_{=0})\cj{n_1} + n_1(\underbrace{m_1\cj{m_1}}_{=0}) = 0.
  \end{gather*}
  Similarly, we see that also $N_1 \coloneqq [n_1] \vee [m_1n_1]$ and
  $N_2 \coloneqq [n_1] \vee [n_1m_1]$ are null lines. Thus, we are in
  the situation depicted in \autoref{fig:null-lines} where we have
  three null quadrilaterals with respective vertices
  \begin{equation*}
    [u_1], [v_1], [u_2], [v_2];\quad
    [m_1n_1], [m_2n_1], [m_2n_2], [m_1n_2];\quad
    [n_1m_1], [n_1m_2], [n_2m_2], [n_2m_1].
  \end{equation*}
  The second and third quadrilateral are different because $m_1$ and
  $n_1$ do not commute (otherwise they would lie on the same line
  through $[1]$ which contradicts the regularity of the quadric
  $\QQ \coloneqq P \cap \SQ$). Our proof will be finished as soon as we
  have shown that the first and the second or the first and the third
  quadrilateral are equal. For this, it is sufficient to show that
  $[v_1] = [m_1n_1]$ or $[v_1] = [n_1m_1]$.

  At first, we argue that the primal part of $[v_1]$ equals the primal
  part of $[m_1n_1]$ or of $[n_1m_1]$. Because $P$ does not intersect
  $\EG$, the projection on the primal part is a regular projectivity
  $P \to [\H]$ with centre $\EG$. We denote projected objects by a
  prime, that is, we write $u'_1$, $v'_1$, $m'_1$ etc., for the primal
  parts of $u_1$, $v_1$, $m_1$ etc. The quadric $\QQ$ is regular and
  ruled and so is its primal projection $\QQ'$. Hence, the point
  $[m'_1] \in \QQ'$ is incident with precisely two lines contained in
  $\QQ'$, say $M'_1$ and $M'_2$. But $[m'_1] \vee [v'_1]$ is contained
  in $\QQ'$. Hence $[v'_1] \in M'_1$ or $[v'_1] \in M'_2$.  Similarly,
  $[v'_1]$ is also incident with one of the two lines contained in
  $\QQ'$ and incident with $[n'_1]$. Thus, $[v'_1] = [m'_1n'_1]$ or
  $[v'_1] = [n'_1m'_1]$.

  Now we have to lift this result to the dual part and show that
  $[v_1] = [m_1n_1]$ or $[v_1] = [n_1m_1]$. The alternative being
  similar, we assume $[v'_1] = [m'_1n'_1]$. This means that we have
  two null quadrilaterals, one with vertices $[u_1]$, $[v_1]$,
  $[u_2]$, $[v_2]$ and one with vertices $[m_1n_1]$, $[m_2n_1]$,
  $[m_2n_2]$, $[m_1n_2]$ such that their primal projections are equal
  and corresponding sides intersect in the vertices $[m_1]$, $[n_1]$,
  $[m_2]$, $[n_2]$ of a planar quadrilateral. But then \autoref{lem:5}
  in \autoref{sec:appendix} asserts their equality and
  $[v_1] = [n_1m_1]$ follows.
\end{proof}

\section{Synthesis procedures}
\label{sec:synthesis}

In this section we discuss the exact synthesis of 5R linkages, based
on our geometric characterisation of 2R spaces. The attentive reader
will note that we neither provide a synthesis algorithm nor a concrete
example. The reason for this (and one of the main points of this
paper) is that the synthesis of 5R linkages can be treated as special
case of a recently introduced synthesis procedure for 6R linkages
\citep{hegedus14:_four_pose_synthesis}. In this context, 2R spaces
play a crucial role.

At first, we re-prove a classical result of \cite{suh69} on the
synthesis of 2R dyads to three poses. Its proof is just a variant of
our computations in the proof of \autoref{th:1}.

\begin{theorem}
  \label{th:2}
  Three poses $[p_0] = [1]$, $[p_1]$, $[p_2] \in \SQ$ that span a
  plane that is not tangent to $\SQ$ and does not intersect $\EG$ are
  incident with precisely two 2R spaces.
\end{theorem}

\begin{proof}
  By \autoref{th:1}, the sought 2R space $P$ contains a null
  quadrilateral. We compute its vertices $[u_1]$, $[v_1]$, $[u_2]$,
  $[v_2]$ as follows. The plane $[p_0] \vee [p_1] \vee [p_2]$ and the
  Study quadric $\SQ$ have a regular conic $C$ in common. This conic
  intersects the null cone in two pairs $m_1$, $m_2$, and $n_1$, $n_2$
  of conjugate complex points that can be computed from the roots of a
  univariate quartic polynomial. Again, we denote projection on the
  primal part by a prime. We can determine the primal part $u'_1$ of
  $u_1$ by solving the system
  $u'_1\cj{u'_1} = u'_1\cj{m'_1} + m'_1\cj{u'_1} = u'_1\cj{n'_2} +
  u'_1\cj{n'_2} =0$.
  It consists of two linear and one quadratic equation. Hence, there
  are precisely two solutions in projective sense. Picking one of
  them, the primal parts of the remaining vertices can be computed
  unambiguously: Write, for example, $v'_1 = u'_1 + \lambda m'_1$ and
  solve $v'_1\cj{n'_1} + n'_1\cj{v'_1} = 0$ for $\lambda$. Finally,
  from $[u'_1]$, $[v'_1]$, $[u'_2]$, $[v'_2]$ the points $[u_1]$,
  $[v_1]$, $[u_2]$, $[v_2]$ can be uniquely recovered as in the proof
  of \autoref{lem:5}.
\end{proof}

\begin{remark}
  The two solutions in \autoref{th:2} both contain the conic $C$. This
  conic is the kinematic image of the coupler motion of a Bennett
  linkage \citep{hamann11}. Thus, as already mentioned by
  \cite{suh69}, the two 2R dyads of \autoref{th:2} can be combined to
  form a Bennett linkage. Moreover, the computation of \autoref{lem:3}
  shows that the two 2R spaces are
  $[1] \vee [h_1] \vee [h_2] \vee [h_1h_2]$ and
  $[1] \vee [h_1] \vee [h_2] \vee [h_2h_1]$ with suitable rotation
  quaternions $h_1$, $h_2$. In other words, they differ only by the
  \emph{ordering of their axes.} This shows that the figures formed by
  opposite axes in a Bennett linkage are congruent.
\end{remark}

In a 5R linkage there are two possibilities to assign a coupler to a
fixed base. The two choices can be distinguished by the degree of
their relative motion with respect to the base. As has been shown
recently by \cite[Theorem~6]{hegedus13:_bonds2}, the degrees of
relative motions between non-adjacent links in the dual quaternion
model of rigid body kinematics are two and three, respectively. In
order to gain additional degrees of freedoms over the synthesis of
Bennett linkages, it is mandatory to use the link to the degree three
motion as coupler. In this sense, we can say that the coupler motion
of a 5R linkage admits a rational cubic parameterisation (a cubic
motion polynomial) in the dual quaternion model of~$\SE$. However, not
every twisted cubic in the Study quadric gives rise to a 5R
linkage. The necessary and sufficient condition is that the coupler
motion can also be generated by a 2R dyad, that is, the twisted cubic
lies in a 2R space.

Using the factorisation theory for rational motions
\citep{hegedus13:_factorization2}, a generic cubic motion polynomial
$C$ can be written, in six different ways, as
$C = (t-h_1)(t-h_2)(t-h_3)$ with linear motion polynomials $t-h_i$,
$i \in \{1,2,3\}$, that parameterise rotations about fixed axes.
Suitable combinations of such factorisations give an overconstrained
6R linkage whose coupler follows the prescribed cubic motion. For
details on how to compute cubic interpolants on quadrics, how to
factor the resulting motions and how to pick factorisations suitable
for linkage synthesis we refer to
\cite{hegedus14:_four_pose_synthesis}. We just highlight the major
differences to the case of 5R linkage synthesis.

By a fundamental property of the underlying theory, the factorisations
of a cubic motion polynomial are, at least in generic situations, in
bijection with the permutations of the real quadratic factors of
$C\cj{C}$ \citep{hegedus13:_factorization2}. Pairings for
``admissible'' kinematic chains correspond to permutations with
different factors at the beginning or at the end. In the case of 5R
synthesis, the situation is special. At least one factorisation is of
the shape $C = (t-h_1)(t-h_2)(t-h_3)$ where either the axes of $h_1$
and $h_2$ or of $h_2$ and $h_3$ are identical. Lets assume the first
case. Then $h_1$ and $h_2$ commute so that $C = (t-h_2)(t-h_1)(t-h_3)$
is another factorisation, different in algebraic sense but identical
in kinematic sense. It can be paired with \emph{just two} of the
remaining four factorisations to form a closed 5R linkage. If $h_2$
and $h_3$ have identical axis, above discussion can be repeated with
$\cj{C} = (t-\cj{h_3})(t-\cj{h_2})(t-\cj{h_1})$ instead of $C$.

For a more detailed discussion of some properties of 5R linkages in
this context, we need some results of twisted cubics that can be found
in classical texts \citep[for example][]{salmon82,cayley85}. The
intersection of a 2R space with the Study quadric is a regular ruled
quadric $\QQ$. It carries two families of rulings. One family contains
a ruling corresponding, via Study's kinematic mapping, to all
rotations about the axis at the base of the 2R dyad, the other family
contains a ruling corresponding to the moving revolute axes. We may
accordingly speak of a ``first'' and a ``second'' family of
rulings. The twisted cubics on $\QQ$ can be partitioned into two
classes: Members of the first class intersect every ruling of the
first family in two and every ruling of the second family in just one
point. For cubics of second class, the situation is just the other way
round. Finally, four points on $\QQ$ determine two one-parametric
families of twisted cubics, one of first class and one of second. A
member of each family is uniquely determined by the choice of one
further point on a ruling through one of the four given points.

These considerations clearly show that four poses in general position
cannot be reached by a 5R linkage unless their kinematic images span a
2R space. On the other hand, three poses in general position already
determine, by \autoref{th:2}, two 2R spaces but also a Bennett
linkage. By above discussion, we have three degrees of freedom to find
a cubic curve through three quadric points. In other words, the number
of free degrees in 5R linkage synthesis exceed that of 4R linkages by
three. Unfortunately, only restricted use can be made of these. It is,
for example, not possible to arbitrarily prescribe one further
orientation: While there is a dual quaternion with prescribed primal
part (orientation) in each 2R space, it will, in general, not lie on
the Study quadric. An additional prescribed position is prevented for
similar reasons.

Finally, also the appearance of two types of factorisations with
commuting factors on the left ($h_1$, $h_2$) or on the right ($h_2$,
$h_3$) can be explained by the two families of twisted cubics on
$\QQ$. If the commuting factors are on the left,
$C_0(t) = (t - h_1)(t-h_2)$ is a quadratic parameterisation of the
line $[1] \vee [h_1]$ and every point on that line corresponds to two
parameter values. In particular, there exist a second parameter value
$t_0$ (besides $\infty$) such that $[C_0(t_0)] = [1]$. But then
$C(t_0)$ yields a second curve point (besides $[1]$) on
$[1] \vee [h_3]$, that is, the cubic belongs to the second family. In
similar manner we see that the cubic belongs to the first family, if
$h_2$ and $h_3$ commute.

\section{Conclusion}
\label{sec:conclusion}

We characterised the kinematic image of 2R dyads and used it to
specialise a recently developed synthesis procedure for 6R linkages in
such a way that it yields 5R linkages. Even if sufficient degrees of
freedom are available, it is not possible in general to find a 5R
linkage that visits three and a half poses (three poses plus one
position or one orientation) because of geometric obstructions. Thus,
a natural selection of interpolation data for the envisaged coupler
motion is not entirely clear. Hopefully, future research will provide
us with more ideas in this direction.

\appendix
\section{Auxiliary results}
\label{sec:appendix}

Here, we prove two technical results. The formulation of
\autoref{lem:4} could be simplified but in its present form its
application in the proof of \autoref{lem:5} is apparent.

\begin{lemma}
  \label{lem:4}
  Given a three-space $E \subset P^7$ and four points $[u'_1]$,
  $[v'_1]$, $[u'_2]$, $[v'_2]$ that span a three-space $F \subset P^7$
  with $E \cap F = \varnothing$, set $U_1 \coloneqq [u'_1] \vee E$,
  $V_1 \coloneqq [v'_1] \vee E$, $U_2 \coloneqq [u'_2] \vee E$,
  $V_2 \coloneqq [v'_2] \vee E$ and consider four projections
  \begin{equation*}
    \mu_1\colon U_1 \to V_1,\quad
    \nu_1\colon V_1 \to U_2,\quad
    \mu_2\colon U_2 \to V_2,\quad
    \nu_2\colon V_2 \to U_1
  \end{equation*}
  with respective centres $[m_1]$, $[n_1]$, $[m_2]$, and $[n_2]$. If
  the centres span a plane $L$, the composition
  $\kappa\coloneqq\nu_2 \circ \mu_2 \circ \nu_1 \circ \mu_1$ of the
  four projections is the identity.
\end{lemma}

\begin{proof}
  Note that $\dim U_1 \vee V_1 = 5$ so that the projection $\mu_1$
  (and also $\nu_1$, $\mu_2$, $\nu_2$) from a suitable point is well
  defined. Take an arbitrary point $[u_1] \in U_1$ but not in $E$ and
  not in $L$ and set $[v_1] \coloneqq \mu_1([u_1])$,
  $[u_2] \coloneqq \nu_1([v_2])$, $[v_2] \coloneqq \mu_2([u_2])$.
  Because these points are not coplanar, we may take them, in above
  order, as base points of a projective coordinate system in
  $G \coloneqq [u_1] \vee [v_1] \vee [u_2] \vee [v_2]$ and complement
  them by four further base points in $E$ and a suitable unit point to
  a projective coordinate system of $P^7$. We select the unit point
  such that its projection into $G$ from $E$ gives
  $([m_1] \vee [m_2]) \cap ([n_1] \vee [n_2])$ and denote induced
  projective coordinates in $U_1$, $V_1$, $U_2$, $V_2$, and $G$
  (obtained by dropping three or four zero coordinates) by writing the
  subspace as subscript. With this convention, the projection centres
  are $[m_1] = [1,1,0,0]_G$, $[n_1] = [0,1,1,0]_G$,
  $[m_1] = [0,0,1,1]_G$, $[n_2] = [1,0,0,1]_G$.  Moreover, we denote
  the coordinate vector of an arbitrary point $x \in U_1$ by
  $[x_0,\star]_{U_1}$ where the star abbreviates four fixed
  coordinates. We then have
  \begin{equation*}
    [x_0,\star]_{U_1} \stackrel{\mu_1}{\mapsto}
    [-x_0,\star]_{V_1} \stackrel{\nu_1}{\mapsto}
    [x_0,\star]_{U_2} \stackrel{\mu_2}{\mapsto}
    [-x_0,\star]_{V_2} \stackrel{\nu_2}{\mapsto}
    [x_0,\star]_{U_1}
  \end{equation*}
  and the claim follows.
\end{proof}

\begin{lemma}
  \label{lem:5}
  Let $E$, $F \subset P^7$ be non-intersecting three spaces, the
  latter spanned by points $[u'_1]$, $[v'_1]$, $[u'_2]$,
  $[v'_2]$. Consider further a regular quadric $\QQ$ containing $E$
  and four vertices $[m_1]$, $[n_1]$, $[m_2]$, $[n_2] \in \QQ$ of a
  planar quadrilateral. Then there exists a unique spatial
  quadrilateral with vertices $[u_1]$, $[v_1]$, $[u_2]$, $[v_2]$ that
  is contained in $Q$, whose sides $[u_1] \vee [v_1]$,
  $[v_1] \vee [u_2]$, $[u_2] \vee [v_2]$, $[v_2] \vee [u_1]$ are, in
  that order, incident with $[m_1]$, $[n_1]$, $[m_2]$, $[n_2]$ and
  whose projections from $E$ into $F$ are $[u'_1]$, $[v'_1]$,
  $[u'_2]$, and~$[v'_2]$.
\end{lemma}

\begin{proof}
  Denote a quadratic form associated to $\QQ$ by $\qf$.  Given
  $[u'_1]$, $[v'_1]$, $[u'_2]$, $[v'_2]$, we have to reconstruct
  $[u_1]$, $[v_1]$, $[u_2]$, $[v_2]$ subject to the constraints
  \begin{equation}
    \label{eq:4}
    \qf(u_i,u_i) = \qf(v_i,v_i) = \qf(u_i,v_j) = 0, \quad i,j \in \{1,2\}
  \end{equation}
  and
  \begin{equation}
    \label{eq:5}
    [m_1] \in [u_1] \vee [v_1],\quad
    [n_1] \in [v_1] \vee [u_2],\quad
    [m_2] \in [u_2] \vee [v_2],\quad
    [n_2] \in [v_2] \vee [u_1].
  \end{equation}
  If $[u_1]$ is given, we find $[v_1]$ by projecting $[u_1]$ from
  centre $[m_1]$ onto $E \vee [v'_1]$. \autoref{lem:4} tells us that
  we can find $[u_2]$ and $[v_2]$ in similar manner such that
  \eqref{eq:5} is satisfied. Then, some of the conditions in
  \eqref{eq:4} become redundant and it is sufficient to consider only
  \begin{equation}
    \label{eq:6}
    \qf(u_1,u_1) = \qf(v_1,v_1) = \qf(u_2,u_2) = \qf(v_2,v_2) = 0.
  \end{equation}
  In a projective coordinate system with base points $[u'_1]$,
  $[v'_1]$, $[u'_2]$, $[v'_2] \in F$ and further base points in $E$
  the quadratic form $\qf$ (and the quadric $\QQ$) is
  described by a matrix of the shape
  \begin{equation*}
    \begin{bmatrix}
      A & B \\
      B & O
    \end{bmatrix}
  \end{equation*}
  where $A$, $B$ and $O$ are matrices of dimension $4 \times 4$, $O$
  is the zero matrix and $B$ is regular. Now \eqref{eq:6} gives rise
  to a linear system for the unknown coordinates of $[u_1]$, $[v_1]$,
  $[u_2]$, and $[v_2]$. The matrix of the linear system is equivalent
  to the matrix $B$ by elementary row transformations, whence
  existence and uniqueness of a solution hinges solely on the
  regularity of~$B$.
\end{proof}

\section*{Acknowledgement}
\label{sec:acknowledgment}

This work was supported by the Austrian Science Fund (FWF): P\;26607
(Algebraic Methods in Kinematics: Motion Factorisation and Bond
Theory).

\bibliographystyle{wow}

\begin{thebibliography}{10}
\providecommand{\natexlab}[1]{#1}
\providecommand{\url}[1]{\texttt{#1}}
\providecommand{\urlprefix}{URL }
\providecommand{\eprint}[2][]{\url{#2}}

\bibitem[{Brunnthaler et~al.(2005)Brunnthaler, Schröcker \&
  Husty}]{brunnthaler05:_bennet_synthesis}
\textsc{Brunnthaler, K., Schröcker, H.P. \& Husty, M.}, 2005 A new method for
  the synthesis of {Bennett} mechanisms. In \emph{Proceedings of CK 2005,
  International Workshop on Computational Kinematics}, Cassino.

\bibitem[{Cayley(1885)}]{cayley85}
\textsc{Cayley, A.}, 1885 On the twisted cubics upon a quadric surface.
  \emph{Messanger of Mathematics} \textbf{14}, 129--132.

\bibitem[{Hamann(2011)}]{hamann11}
\textsc{Hamann, M.}, 2011 Line-symmetric motions with respect to reguli.
  \emph{Mech. Mach. Theory} \textbf{46}(7), 960--974.

\bibitem[{Hegedüs et~al.(2013{\natexlab{a}})Hegedüs, Schicho \&
  Schröcker}]{hegedus13:_factorization2}
\textsc{Hegedüs, G., Schicho, J. \& Schröcker, H.P.}, 2013{\natexlab{a}}
  Factorization of rational curves in the {Study} quadric and revolute
  linkages. \emph{Mech. Mach. Theory} \textbf{69}(1), 142--152.

\bibitem[{Hegedüs et~al.(2013{\natexlab{b}})Hegedüs, Schicho \&
  Schröcker}]{hegedus13:_bonds2}
\textsc{Hegedüs, G., Schicho, J. \& Schröcker, H.P.}, 2013{\natexlab{b}} The
  theory of bonds: A new method for the analysis of linkages. \emph{Mech. Mach.
  Theory} \textbf{70}, 407--424.

\bibitem[{Hegedüs et~al.(2014)Hegedüs, Schicho \&
  Schröcker}]{hegedus14:_four_pose_synthesis}
\textsc{Hegedüs, G., Schicho, J. \& Schröcker, H.P.}, 2014 Four-pose
  synthesis of angle-symmetric {6R} linkages. Accepted for publication in ASME
  J. Mechanisms Robotics.

\bibitem[{Klawitter(2015)}]{klawitter15}
\textsc{Klawitter, D.}, 2015 \emph{Clifford Algebras. Geometric Modelling and
  Chain Geometries with Application in Kinematics}. Springer Spektrum.

\bibitem[{Salmon(1882)}]{salmon82}
\textsc{Salmon, G.}, 1882 \emph{A treatise on the analytic geometry of three
  dimensions}. Hodges, Figgis \& Co., Dublin, 4 edn.

\bibitem[{Selig(2005)}]{selig05:_geometric_fundamentals_robotics}
\textsc{Selig, J.}, 2005 \emph{Geometric Fundamentals of Robotics}. Monographs
  in Computer Science, Springer, 2 edn.

\bibitem[{Suh(1969)}]{suh69}
\textsc{Suh, C.H.}, 1969 On the duality in the existence of {R-R} links for
  three positions. \emph{ASME J. Mechanical Design} \textbf{91}(1), 129--134.

\end{thebibliography}

\end{document}